%% file: sample.tex
\PassOptionsToPackage{final}{graphicx}
\documentclass[12pt,a4paper]{article}

\title{Weighted Dyck paths for nonstationary queues}
\date{\today}
\author{G.~Bet\footnote{Università degli Studi di Firenze}, J.~Selen\footnote{ASML}, A.~Zocca\footnote{Vrije Universiteit Amsterdam}}

\usepackage[english]{babel}

\usepackage{geometry}
\usepackage{amsmath}
\numberwithin{equation}{section} 

\usepackage{enumitem}

\usepackage[mode = buildnew]{standalone}

\usepackage{subfig}
\usepackage{graphicx}

\usepackage{todonotes}

\usepackage{hyperref}

\usepackage{ulem}

\input{TikZSettingsJori.tex}

\input{MacrosJori.tex}

\newcommand{\DG}{$\Delta_{(i)}/G/1$~}
\begin{document}
\maketitle

\begin{abstract}
We consider a model for a queue in which only a fixed number $N$ of customers can join. Each customer joins the queue independently at an exponentially distributed time. Assuming further that the service times are independent and follow an exponential distribution, this system can be described as a two-dimensional Markov process on a finite triangular region $\stateSpace$ of the square lattice. We interpret the resulting random walk on $\stateSpace$ as a Dyck path that is weighted according to some state-dependent transition probabilities that are constant along one axis, but are rather general otherwise. We untangle the resulting intricate combinatorial structure by introducing appropriate generating functions that exploit the recursive structure of the model. This allows us to derive a fully explicit expression for the probability density function of the number of customers served in any busy period (equivalently, of the length of any excursion of the Dyck path above the diagonal) as a weighted sum with alternating sign over a certain subclass of Dyck paths, whose study is of independent interest.
\end{abstract}

\section{Introduction}%
Time-dependent queueing models are powerful tools for the analysis of real-life situations where the long-term behaviour of a system is not a good approximation for its performance. Examples of applications include call centers \cite{brown2005statistical} and outpatient wards of hospitals where the server operates only over a finite amount of time \cite{kim2014callcenter, kim2014choosing}. On the other hand, rigorous and explicit results on time-dependent models are mostly out of reach because the standard tools of renewal theory and ergodic theory are often not applicable. In this paper we focus on a certain class of time-dependent models called \textit{transitory queueing systems}, introduced in \cite{honnappa2014transitory}, and defined as systems that operate only during a finite time horizon. Thus only the time-dependent behavior is of interest. Hence transitory queueing systems are time-dependent models that present even greater technical challenges because their steady-state distribution is trivial (all the probability mass is concentrated in zero). One common approach to tackle this issue is to introduce a scaling parameter $N$ in the queueing model and approximate the resulting system with the asymptotic model obtained by taking $N\to\infty$. This approximation is justified in terms of stochastic-process limits, see e.g., \cite{StochasticProcess, whitt2010queues} and references therein. This approach is robust because it relies on a functional Central Limit Theorem and it has proven to be highly successful. However, this approach has two drawbacks. First, the asymptotic results yield precise approximations only for very large $N$, and often accurate error estimates are not available. Second, the asymptotic model is often still too complicated to be analyzed exactly, and thus further approximations are needed. In this paper we aim at developing novel tools for the analysis of transient queueing systems that do not rely on \textit{any} approximation scheme and that provide \textit{explicit} formulas for the relevant performance metrics. We emphasize that our approach is not meant to replace the classical asymptotic approximation scheme, but rather to complement it when the approximations it provides are unreliable or analytically intractable.

The canonical model for the study of transitory queueing systems is the so-called \DG model \cite{honnappa2015delta,honnappa2014transitory} in which a single queue serves a finite pool of $N$ potential customers, where $N$ will be fixed throughout this paper. Each customer joins the queue at a time $T_i$, where $(T_i)_{i=1}^N$ are positive i.i.d.~random variables. Once in the queue, customers are served in a first-come-first-served fashion. Each customer requires an amount of service $S_i$, where are i.i.d.~random variables which are independent from the $T_i$. Once a customer is served, they leave the system permanently. The \DG model was first introduced in \cite{honnappa2015strategic}, where it emerged as the solution of a game-theoretic optimization problem in a queueing setting. Furthermore, in \cite{honnappa2014transitory} it was proven that, under the appropriate scaling, several other transitory models have the same asymptotic behavior as the \DG model. Hence, the \DG model should be seen as the canonical transitory queueing model, similarly as how the $G/G/1$ queue is the canonical stationary queueing model. The asymptotic regime $N\to\infty$ of the \DG queue has been studied extensively in recent years. In \cite{honnappa2015delta} the authors prove a functional Law of Large Numbers (fLLN) and a functional Central Limit Theorem (fCLT) for the queue-length process. They identify the limit processes explicitely, but these are considerably difficult to analyze and explicit formulas for quantities of interest are not available. In a series of works \cite{bet2016alternative,bet2014heavy,bet2015finite,bet2016when} the authors consider the \DG queue in the heavy-traffic regime that is obtained by assuming the instant of peak congestion is at $t=0$. Their results are also fCLT's for the queue-length process. In all the cases, the limit process is a reflected stochastic process with negative quadratic drift, for which several explicit expressions for quantities of interest are available, see \cite{bet2014heavy} for details. 

Here we offer a new perspective on the \DG model, which we now summarize. We assume that the arrival times $T_i$ are exponentially distributed with rate $\lambda$, and that the service times $S_i$ are exponentially distributed with mean $1/\mu$.  We focus on the embedded Markov chain associated to the queueing process, and we show that the path of the Markov chain is a Dyck path of order $N$, that is, a staircase walk in $\Nat^2$ from $(0,0)$ to $(N,N)$ that stays above (but may touch) the diagonal. It follows that the transition probabilities of the Markov chain induce a probability measure on the space of Dyck paths. Our result is then an explicit expression for the probability density function of the excursion lengths of the Dyck path above the diagonal as a weighted sum over a certain subclass of Dyck paths that, roughly speaking, do not avoid the diagonal. Furthermore, we show that our result holds for general transition probabilities that include the transition probabilities associated with the \DG model.

Dyck paths are some of the most well-studied objects in combinatorics and thus the literature on the subject is vast. Perhaps closest to our approach is the work of Viennot \cite{viennot1985combinatorial}. That paper finds general relationships between a certain class of orthogonal polynomials and weighted Motzkin paths, which are a generalization of Dyck paths that allow for diagonal jumps. In particular, it shows that the elements of the inverse coefficient matrix of the polynomials are related to the sum of the weights of all Motzkin paths starting in $(0,0)$ and with varying length and endpoint. This is in line with our proof technique for Proposition \ref{prop:A_inverse_explicit_expression}. The authors in \cite{luczak2004building} provide a probabilistic procedure to iteratively grow certain general combinatorial structures $(T_k)_{k=1}^{\infty}$ in such a way that at each step the law of $T_k$ is uniform among all possible such structures of size $k$. Similarly, in our model a random Dyck path of order $N$ is generated via a local mechanism, i.e., by giving transition probabilities at each lattice site.

The rest of the paper is organized as follows. In Section \ref{sec:model_description} we define the \DG model formally and we state our main result. In Section \ref{sec:distribution_first_busy_period} we prove our main result by first developing a recursion for the distribution of the number of customers served in the first busy period, and then solving the recursion explicitely. 

\section{Model description, Dyck paths and main result}%
\label{sec:model_description}%
Consider a single-server queue that serves customers in a first-come first-served manner. There is a finite pool of $N$ customers, each of which enters the system only once. Each customer independently joins the queue after an exponential time with rate $\la$ and requires a service time that is exponentially distributed with rate $\mu$. For notational convenience we denote by 
\begin{align}%
\la_n \defi \la (N - n)
\end{align}%
the arrival rate of customers to the system if $n$ customers have already arrived to the system.

The state of the system at time $t\geq 0$ is described by a vector $X(t) \defi (X_1(t),X_2(t)) \in \Nat^2$ where $X_1(t)$ is the number of completed services at time $t$ and $X_2(t)$ is the number of customers that have joined the system up until time $t$. In view of our assumptions, the process $\{ X(t) \}_{t \ge 0}$ is a Markov process on the state space
\begin{equation}%
\stateSpace \defi \{ (i,j) \in \Nat^2 : 0 \le i \le N, ~ 0 \le j \le i \}.
\end{equation}
The transition rate diagram is depicted in Figure~\ref{fig:transition_rate_diagram}. The Markov process $\{ X(t) \}_{t \ge 0}$ is clearly reducible and admits the trivial equilibrium distribution $\pi$ with $\pi_{N,N}=1$ and $\pi_{i,j}=0$ otherwise.

\begin{figure}[!hbt]%
\centering%
\includegraphics{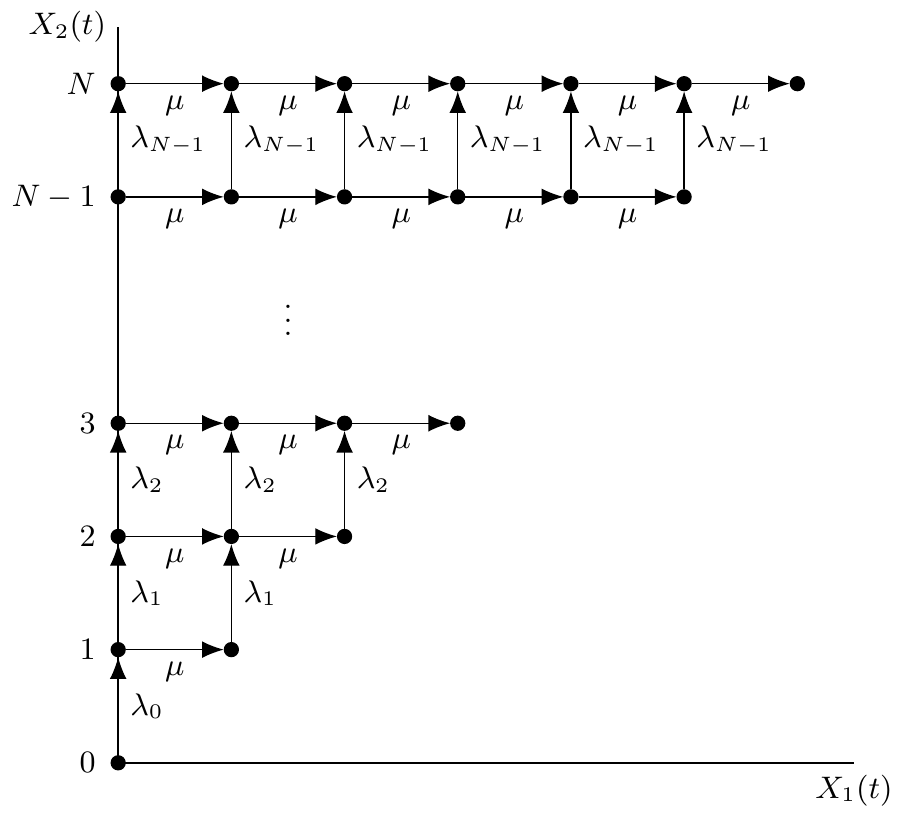}%
\caption{{Transition rate diagram of the Markov process $\{ X(t) \}_{t \ge 0}$.}}%
\label{fig:transition_rate_diagram}%
\end{figure}%

As illustrated in Figure~\ref{fig:transition_rate_diagram}, the state space $\stateSpace$ is highly structured. Our approach crucially leverages this structure. 
We refer to the set of states in the $j$-th row of $\stateSpace$
\begin{align}%
\set{P}_j \defi \{ (0,j),(1,j),\ldots,(j,j) \}
\end{align}%
as the $j$-th \textit{phase}, which corresponds to the situation in which exactly $j$ customers have arrived in the system. 
We denote the collection of diagonal states as $\set{D}_0 \defi \{ (1,1),(2,2),\ldots,(N,N) \}$, and further use the notation $\set{D}_n \defi \{ (0,n),(1,n + 1),\ldots,(N - n,N) \},$ $1\leq n \leq N$ to denote the set of states on the $n$-th superdiagonal of $\stateSpace$. 

It does not seem possible to find an explicit solution for the Kolmogorov equations associated to $X(t)$ due to the time-inhomogeneous arrival process. Therefore we study the associated embedded Markov chain on $\stateSpace$, which we denote, with an abuse of notation, as $(X(k))_{k=0}^{2N}$. Conditionally on $X(k) = (i,j)$ with $i<j$, we have
\begin{align}\label{eq:transition_probabilities_embedded_Markov_chain}%
X(k+1) = 
\begin{cases}
(i+1,j) & \text{with probability } \rho_j \\
(i,j+1) & \text{with probability } 1-\rho_j,
\end{cases}
\end{align}%
where
\begin{equation}\label{eq:rho_n_definition}%
\rho_j \defi \frac{\mu}{\mu + \la_j}.
\end{equation}%
In terms of the queueing system, $\rho_j$ is the probability that a service occurs before an arrival when $j$ customers have already arrived, but not all of them have already been served. Note that, conditionally on $X(k) = (i,i)$, we have $X(k+1) = (i,i+1)$ with probability one. The \DG queueing model corresponds to the choice $\rho_j=0$ if $j=0$ and $\rho_j= \mu/(\mu+\la_j)$ if $j=1,\ldots,N$. We focus on the random variable $S$ describing the number of customers served in the first busy period, which is the time between the instant a customer arrives to an empty system and the instant a customer departs the system leaving behind an empty system. Our main result is an explicit expression for the probability $s_i$ that exactly $i$ customers are served in the first busy period, i.e., $s_i \defi \Prob{S = i}$. 

From the discussion above it follows that the trajectory of the Markov chain is a Dyck path of order $N$. We denote the set of Dyck paths of order $N$ as $\mathfrak D_N$. A Dyck path $u\in\mathfrak D_N$ is fully characterized by the sequence $(u_j)_{j=1}^N$ of jumps to the right at each of the \textit{phases} $\set P_j$, with $j=1,\ldots, N$. With an abuse of notation we write
\begin{align}%
\mathfrak D_N = \Big\{(u_1,\ldots, u_N) \in\Nat^N : \sum_{j=1}^k u_j \leq k\text{ for all }k=1,\ldots, N-1,\text{ and }\sum_{j=1}^N u_j = N\Big\}.
\end{align}%
The transition probabilities \eqref{eq:transition_probabilities_embedded_Markov_chain} induce a probability measure $\bar{\mathbb P}$ on $\mathfrak D_N$ such that,
\begin{align}\label{eq:induced_probability_on_space_of_Dyck_paths}%
\bar{\mathbb P}(u) = \prod_{j=1}^N \rho_j^{u_j} (1-\rho_j)^{\mathds 1_{\{\sum_{i=1}^ju_i<j\}}},\quad u=(u_1,\ldots, u_N)\in\mathfrak D_N.
\end{align}%
From a probabilistic perspective, equation \eqref{eq:induced_probability_on_space_of_Dyck_paths} can be understood as follows: the probability that the Markov chain jumps $u_j$ times to the right at phase $\set P_j$ is $\rho_j^{u_j}$. Moreover, if $\sum_{i=1}^ju_i=j$, then the Markov chain hits the diagonal on $(j,j)$ and in that case it jumps up with probability one. Otherwise, it jumps up with probability $1-\rho_j$. From a combinatorial perspective, $\rho_j$ and $1-\rho_j$ may be interpreted as \textit{weights} associated to their respective edges in $\stateSpace$. Equation \eqref{eq:induced_probability_on_space_of_Dyck_paths} then assigns to the Dyck path $u$ a weight $w(u) \defi \bar{\mathbb P}(u)$, which is simply the product of the weights of the edges it traverses.

Equation \eqref{eq:induced_probability_on_space_of_Dyck_paths} suggests partitioning the state space $\stateSpace$ in the $N$ phases $\set P_1,\ldots, \set P_N$ in order to study the probability measure $\bar{\mathbb P}$. Crucially, the $(j+1)-$th phase may only be reached from the $j$-th phase and the transition probabilities between $\set P_j$ and $\set P_{j+1}$ only depend on $j$. We exploit this recursive structure by associating to each phase a generating function $P_j(z)$ and then expressing $P_{j+1}(z)$ in terms of $P_{j}(z)$. We then obtain the probability density function of the number of customers served in the \textit{first} busy period (equivalently, the probability density function of the length of the first excursion of the associated Dyck path above the diagonal) by computing $P_j(\bar z)$ for some explicit $\bar z\in\Real$. We are able to fully solve this recursion by rewriting it as a linear system of equations and then inverting the coefficients matrix. 

A crucial role in our result will be played by those Dyck paths that hit the diagonal whenever they jump to the right, see Figure \ref{fig:Dyck_path_examples}. We make this precise in terms of the number of right jumps $(u_1,\ldots, u_n)$ of the Dyck path $u$ at each phase $\set P_j$. We define a \textit{feasible allocation} $(u_1,\ldots, u_n)$ in a recursive manner, starting from $u_1$, as follows: $u_1$ is either 1 or 0, then
\begin{enumerate}%
\item If $u_{i-1} = u_{i-2} = \ldots = u_{i-k+1}= 0$, then $u_i$ is either $k$ or 0;
\item If $u_{i-1} \neq 0$, then $u_i$ is either 1 or 0.
\end{enumerate}%
Moreover,  $(u_1,\ldots, u_n)$ is such that $\sum_{i=1}^nu_i=n$.
We denote by $\set{U}_n$ the set of feasible allocations. With a minor abuse of terminology, we refer to elements of $\set U_n$ interchangeabily as feasible allocations and as Dyck paths. The set $\set U_n$ then represents all those Dyck paths of order $n\leq N$ that hit the diagonal whenever they jump to the right. Some examples of feasible allocations for $n = 4$ are $(1,1,0,2)$, $(0,0,3,1)$, $(0,2,0,2)$ and $(0,0,0,4)$. Some examples of unfeasible allocations for $n = 4$ are $(1,0,1,2)$, since $u_3$ must be 0 or 2, $(1,0,0,2)$, since $u_4$ must be 3, and $(0,0,2,2)$, since $u_3$ must be 0 or 3. See Figure \ref{fig:Dyck_path_examples} for an example of both a feasible and an unfeasible allocation in terms of Dyck paths.
\begin{figure}[!hbt]%
\centering%
\includegraphics[scale=0.9]{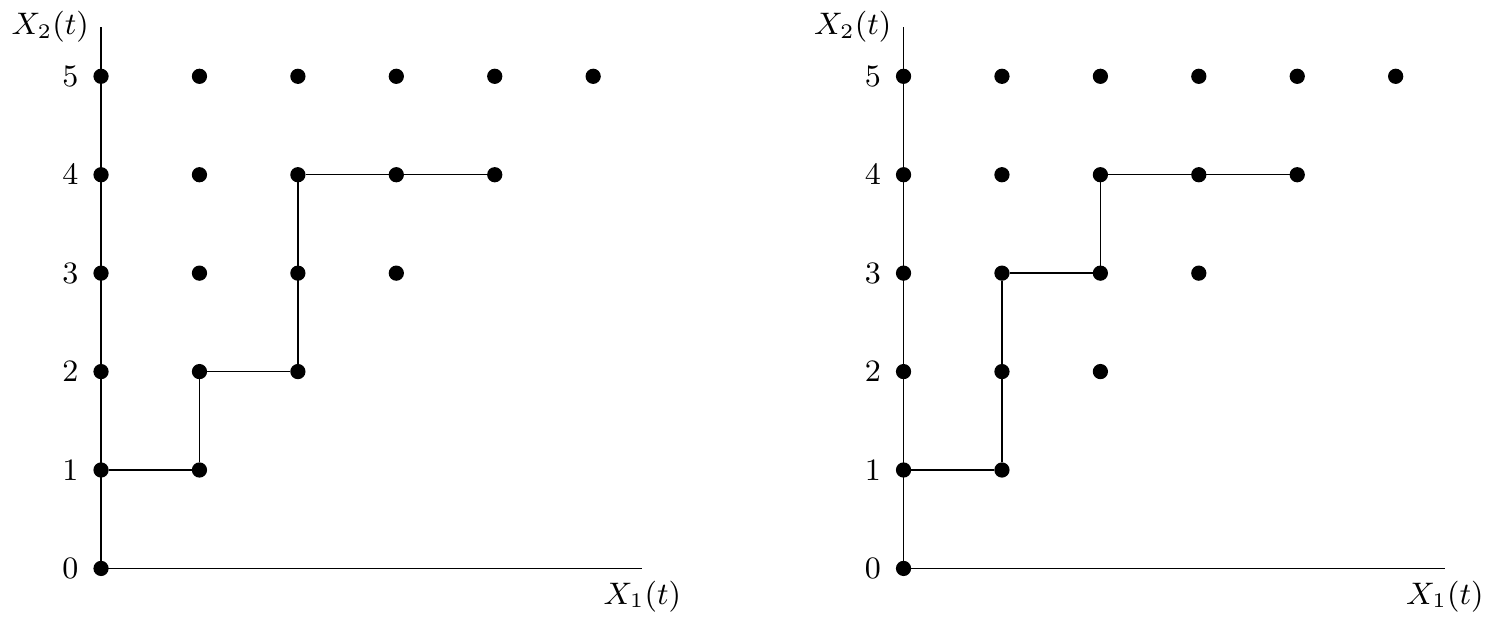}%
\caption{Examples of feasible and unfeasible allocations in $\set{U}_4$ in terms of Dyck paths. The Dyck path on the left corresponds to the feasible allocation $u = (1,1,0,2)$, the one on the right corresponds to the allocation $u = (1,0,1,2)$, which is unfeasible since $u_3$ must be 0 or 2.}%
\label{fig:Dyck_path_examples}%
\end{figure}%
For every Dyck path $u\in \set U_N$ there exists $\set J=\set J(u)\subseteq \{1,\ldots, N\}$ such that \eqref{eq:induced_probability_on_space_of_Dyck_paths} simply reads 
\begin{align}%
\bar{\mathbb P}(u) = \prod_{j\in \set J} \rho_j^{u_j}\prod_{k\in J^{\text c}} (1-\rho_k),
\end{align}%
where $\set J^{\text c} \defi \{1,\ldots,N\}\setminus \set J$. Here the set $\set J$ represents the phases where the Dyck path jumps to the right and hits the diagonal. The set $\set J^{\text c}$ then represents the phases where the Dyck path jumps up without jumping to the right.
Conditioning on the phase in which the path first jumps to the right, it can be shown that
$\vert\,\set{U}_n\vert = 2^{n - 1}$ for  $1\leq n \leq N$. 

We are finally able to state our main result.
\begin{theorem}\label{th:main_theorem}
The probability that $i$ customers are served in the first busy period of the \DG queue or, equivalently, the probability that the corresponding Dyck path hits the diagonal for the first time in $(i,i)$ is given by
\begin{align}\label{eq:s_i_explicit_expression_main_thm}%
s_i = \sum_{(u_1,u_2,\ldots,u_i) \in \set{U}_i}b(\rho_{1}^{u_1},\rho_{2}^{u_2},\ldots,\rho_i^{u_i}) \rho_{1}^{u_1} \rho_{2}^{u_2} \cdots \rho_i^{u_i},
\end{align}%
where $b:\Real^i\to\Real$ is an explicit function defined later in \eqref{eq:b_definition_general} and takes both positive and negative values.
\end{theorem}
From a combinatorial perspective, $s_i$ may be interpreted as the sum of the weights of all those Dyck paths of order $i$ that do not hit the diagonal, which are in bijection with Dyck paths of order $i-1$. Then, equation \eqref{eq:s_i_explicit_expression_main_thm} may be interpreted as a decomposition of the sum of weighted Dyck paths of order $i-1$ in terms of only those weighted Dyck paths that are associated with feasible allocations in $\set{U}_i$ (the right-hand side). 

Let us briefly make explicit the dependence of $s_i$ on the initial number of customers $N$ as $s_i^{(N)}$. Then, conditionally on $S=n$, the probability that $i$  customers are served in the second busy period is $s_i^{(N-n)}$ and, hence, Theorem \ref{th:main_theorem} gives the joint distribution of the number of customers served in \textit{all} busy periods.


\section{The number of customers in the first busy period}%
\label{sec:distribution_first_busy_period}%
We prove Theorem \ref{th:main_theorem} in two steps. First, in Subsection~\ref{subsec:recursion_s_n} we define a generating function $P_j(z)$ associated to phase $j$ and derive a relation between $P_j(z)$ and $P_{j-1}(z)$. The probabilities $s_i$ are obtained by evaluating $P_n(\bar z)$ in a specific $\bar z = \bar z(n)$, yielding a recursive relation for $s_1,\ldots, s_N$. Then, in Subsection~\ref{subsec:solvingrecursion} we interpret this recursive relation as a linear system $A \mathbf s = \mathbf b$, where $\mathbf s = (s_1,\ldots, s_N)$ and $A$ is a lower-triangular matrix. By calculating explicitly the inverse $A^{-1}$, we finally obtain the explicit expression for the probabilities $\mathbf s=(s_1,\ldots, s_N)$ as stated in \eqref{eq:s_i_explicit_expression_main_thm}.


\subsection{Developing a recursion}%
\label{subsec:recursion_s_n}%
We begin by introducing some notation. Given any stochastic process $Y$, we let $\E{ y }{f(Y)}$ represent the expectation of a functional of $Y$, conditional on $Y(0) = y$ and similarly for $\Prob{y}{\cdot}$. For every subset $\set{A} \subsetneq \stateSpace$, the hitting-time $H_{\set{A}}$ is the random variable
\begin{equation}%
H_{\set{A}} \defi \inf\{ t > 0 : \lim_{s \uparrow t} X(s) \neq X(t) \in \set{A} \},
\end{equation}%
which describes the first time that the process $\{ X(t) \}_{t \ge 0}$ started at $(0,0)$ enters the subset $\set{A}$. For a singleton $x \in \stateSpace$, $H_x$ should be understood as $H_{\{ x \}}$.


Let $p_n(i)$ be the probability that, conditionally on the starting point $X(0) = (0,0)$, the Markov process $\{ X(t) \}_{t \ge 0}$ first visits phase $n$ hitting state $(i,n)$ and without residing in $\set{D}_0$, i.e.,
\begin{equation}%
p_n(i) \defi \Prob{(0,0)}{H_{\set{P}_n} < H_{\set{D}_0}, ~ X(H_{\set{P}_n}) = (i,n)}, \quad 0 \le i \le n, ~ 1 \le n \le N.
\end{equation}%
Note that $p_n(n-1) = p_n(n) = 0$ for $2\leq n\leq N$. Define the generating function of the sequence $(p_n(i))_{i=0}^{n-2}$ as
\begin{equation}%
P_n(z) \defi \sum_{i = 0}^{n - 2} p_n(i) z^i, \quad z \in \Complex,~2 \le n \le N.
\end{equation}%
For notational convenience, we also define $P_1(z) \defi 1$. Clearly, if $N = 1$, then $s_1 = 1$, hence from now on we will focus on $N>1$. The strong Markov property implies that $s_1 = \rho_1$, and furthermore
\begin{equation}%
s_n = \sum_{i=0}^{n-2}p_n(i)\rho_n^{n-i} = \rho_n^n P_n(\rho_n^{-1}), \quad 2 \le n \le N, \label{eqn:relation_GF_and_s_n}
\end{equation}%
where $\rho_n$ is defined in \eqref{eq:rho_n_definition}. Note that \eqref{eqn:relation_GF_and_s_n} implies $s_N = P_N(1)$. Equation \eqref{eqn:relation_GF_and_s_n} is the crucial relation that allows us to obtain a recursive expression for the probabilities $(s_n)_{n=1}^N$ starting from a recursive expression for the generating functions $(P_n(\cdot))_{n=1}^N$.

Finally, let $G_p(z)$ denote the probability generating function of a geometric random variable with support $\{ 0,1,\ldots \}$ and success probability $1-p$, i.e.,
\begin{equation}%
  G_{p}(z) \defi \frac{1 - p}{1 - p z}, \quad \vert z\vert < \frac{1}{p}.
\end{equation}%

We are now ready to state our first result.
\begin{lemma}\label{lem:generating_function_explicit}%
For any choice of positive transition probabilities $(\rho_{j})_{j=1}^N$, the generating functions satisfy the recursion
\begin{align}%
P_{n + 1}(z) = G_{\rho_n}(z) \Bigl[ P_n(z) - s_n z^n \Bigr], \quad 1 \le n \le N - 1.
\end{align}%
In particular,
\begin{equation}\label{eqn:P_n+1(z)_explicit}%
P_{n + 1}(z) = \prod_{i = 1}^{n} G_{\rho_i}(z) - \sum_{i = 1}^n s_i z^i \prod_{j = i}^{n} G_{\rho_j}(z), \quad |z| < \frac{1}{\rho_n}. 
\end{equation}%
\end{lemma}%

\begin{proof}%
We start by expressing $P_{n + 1}(z)$ in terms of $P_n(z)$. From the strong Markov property at time $H_{\set{P}_n}$ we can write
\begin{align}%
p_{n + 1}(i) &= \sum_{j = 0}^i p_n(j) \rho_n^{i - j} (1 - \rho_n), \quad 0 \le i \le n - 2, \label{eqn:p_n+1(i)_in_terms_of_p_n(j)} \\
p_{n + 1}(n - 1) &= \sum_{j = 0}^{n - 2} p_n(j) \rho_n^{n - 1 - j} (1 - \rho_n). \label{eqn:p_n+1(n-1)_in_terms_of_p_n(j)}
\end{align}%
Multiply both sides of \eqref{eqn:p_n+1(i)_in_terms_of_p_n(j)} by $z^i$ and sum over all $i$ with $0 \le i \le n - 2$ and multiply both sides of \eqref{eqn:p_n+1(n-1)_in_terms_of_p_n(j)} by $z^{n - 1}$. Sum the two resulting expressions to get
\begin{equation}%
P_{n + 1}(z) = \sum_{i = 0}^{n - 2} \sum_{j = 0}^i p_n(j) \rho_n^{i - j} (1 - \rho_n) z^i + \sum_{j = 0}^{n - 2} p_n(j) \rho_n^{n - 1 - j} (1 - \rho_n) z^{n - 1}.
\end{equation}%
Switch the order of the double summation to obtain
\begin{align}%
P_{n + 1}(z) &= (1 - \rho_n) \Bigl[ \sum_{j = 0}^{n - 2} p_n(j) \sum_{i = j}^{n - 2} \rho_n^{i - j} z^i + \sum_{j = 0}^{n - 2} p_n(j) \rho_n^{n - 1 - j} z^{n - 1} \Bigr] \notag \\
&= (1 - \rho_n) \Bigl[ \sum_{j = 0}^{n - 2} p_n(j) \sum_{k = 0}^{n - 2 - j} \rho_n^k z^{j + k} + \sum_{j = 0}^{n - 2} p_n(j) \rho_n^{n - 1 - j} z^{n - 1} \Bigr].
\end{align}%
The summation over $k$ is a geometric sum. Performing this summation and rewriting yields the recursive expression
\begin{align}%
P_{n + 1}(z) &= (1 - \rho_n) \Bigl[ \sum_{j = 0}^{n - 2} p_n(j) \frac{z^j - \rho_n^{n - 1 - j} z^{n - 1}}{1 - \rho_n z} + \sum_{j = 0}^{n - 2} p_n(j) \rho_n^{n - 1 - j} z^{n - 1} \Bigr] \notag \\
&= \frac{1 - \rho_n}{1 - \rho_n z} \Bigl[ \sum_{j = 0}^{n - 2} p_n(j) z^j - z^n \rho_n^n \sum_{j = 0}^{n - 2} p_n(j) \rho_n^{- j} \Bigr] \notag \\
&= G_{\rho_n}(z) \Bigl[ P_n(z) - s_n z^n \Bigr]. \label{eqn:P_n+1(z)_in_terms_of_P_n(z)}
\end{align}%
To prove the explicit expression \eqref{eqn:P_n+1(z)_explicit} we iterate the recursion \eqref{eqn:P_n+1(z)_in_terms_of_P_n(z)}, obtaining
\begin{equation}%
P_{n + 1}(z) = P_2(z) \prod_{i = 2}^n G_{\rho_i}(z) - \sum_{i = 2}^n s_i z^i \prod_{j = i}^n G_{\rho_j}(z), \label{eqn:P_n+1(z)_in_terms_of_P_2(z)}
\end{equation}%
which we can further simplify by noting that
\begin{equation}%
P_2(z) = p_2(0) = 1 - \rho_1 = (1 - \rho_1 z) G_{\rho_1}(z).
\end{equation}%
Since $s_1 = \rho_1$, we finally obtain \eqref{eqn:P_n+1(z)_explicit}. 
\end{proof}%

Note that for the \DG queue we have $\rho_i^{-1} = (\mu+\la_i)/\mu >(\mu+\la_j)/\mu = \rho_j^{-1}$ for any $i<j$. Therefore, $G_{\rho_{i}}(\rho_n^{-1})$ is well defined for all $i < n$.

In the proof of Lemma \ref{lem:generating_function_explicit} we did not make use of the precise expression of $\rho_n$, and so \eqref{eqn:P_n+1(z)_explicit} still holds when replacing $\lambda_i$ with any sequence of positive decreasing numbers. Combining Lemma \ref{lem:generating_function_explicit} with \eqref{eqn:relation_GF_and_s_n} allows us to obtain a recursive expression for $s_n$. We first present the expression for $s_n$ for a general decreasing sequence $(\la_n)_{n=1}^N$, and then the one obtained when setting $\la_n = \la (N - n)$. We adopt the convention that the empty sum $\sum_{i = 1}^{0} (\cdot) = 0$ and the empty product $\prod_{i = 1}^{0} (\cdot) = 1$. 

\begin{corollary}\label{cor:recursion_s_n}%
Assume $(\la_n)_{n=1}^N$ is a sequence such that ${\la_1>\ldots>\la_{N-1}>\la_N = 0}$. Then, 
\begin{align}\label{eq:general_recursion_s_n}%
s_n = \rho_n^n \prod_{k=1}^{n-1}\frac{\la_k}{\la_k - \la_n} - \sum_{i = 1}^{n - 1} s_i \rho_n^{n - i} \prod_{k=i}^{n-1}\frac{\la_k}{\la_k - \la_n},\quad 2 \le n \le N,
\end{align}%
with initial term $s_1 = \rho_1$. In particular, when $\la_n = \la (N - n)$, the probabilities $s_n$ satisfy the recursion
\begin{equation}\label{eq:recursion_s_n}%
s_n = \rho_n^n \binom{N - 1}{n - 1} - \sum_{i = 1}^{n - 1} s_i \rho_n^{n - i} \binom{N - i}{n - i},
\end{equation}%
with initial term $s_1 = \rho_1$.
\end{corollary}%

\begin{proof}%
Combining the result of Lemma~\ref{lem:generating_function_explicit} with \eqref{eqn:relation_GF_and_s_n} yields the following recursion, for $2 \le n \le N - 1$,
\begin{equation}\label{eq:rewriting_generating_function}%
s_n = \rho_n^n \prod_{i = 1}^{n - 1} G_{\rho_i}(\rho_n^{-1}) - \sum_{i = 1}^{n - 1} s_i \rho_n^{n - i} \prod_{j = i}^{n - 1} G_{\rho_j}(\rho_n^{-1}), \quad s_N = 1 - \sum_{i = 1}^{N - 1} s_i.
\end{equation}%
Note that, by our assumption on the sequence $(\la_n)_{n=1}^N$, we have $\rho_1^{-1} > \cdots > \rho_N^{-1} = 1$. Therefore, $G_{\rho_{i}}(\rho_n^{-1})$ is well defined for all $i < n$. The first expression \eqref{eq:general_recursion_s_n} follows from
\begin{equation}\label{eq:rewriting_generating_function_product}%
G_{\rho_k}(\rho_n^{-1}) = \frac{1 - \rho_k}{1 - \frac{\rho_k}{\rho_n}} = \frac{1 - \frac{\mu}{\mu + \la_k}}{1 - \frac{\mu + \la_n}{\mu + \la_k}} = \frac{\la_k}{\la_k - \la_n}.
\end{equation}%
Moreover, when $\la_n = \la(N - n)$ we get
%
\begin{equation*}%
\prod_{k = l}^{n - 1} G_{\rho_k}(\rho_n^{-1}) = \prod_{k = l}^{n - 1} \frac{N - k}{n - k} = \frac{N - l}{n - l} \frac{N - l - 1}{n - l - 1} \frac{N - l - 2}{n - l - 2} \cdots \frac{N - n + 1}{1} = \binom{N - l}{n - l},
\end{equation*}%
which proves \eqref{eq:recursion_s_n}.
\end{proof}%

\subsection{Solving the recursion}%
\label{subsec:solvingrecursion}%
In this section we solve the recursion \eqref{eq:general_recursion_s_n} to find an explicit expression for $s_n$. Recall that for $n = 1,2,\ldots,N$,
\begin{align}%
s_n = \rho_n^n \prod_{k=1}^{n-1}\frac{\la_k}{\la_k - \la_n} - \sum_{i = 1}^{n - 1} s_i \rho_n^{n - i} \prod_{k=i}^{n-1}\frac{\la_k}{\la_k - \la_n},
\end{align}%
Divide both sides by $\rho_n^n$ and bring all $s_i$ terms to one side to obtain
\begin{equation}%
\sum_{i = 1}^{n} \frac{s_i}{\rho_n^i} \prod_{k=i}^{n-1}\frac{\la_k}{\la_k - \la_n} = \prod_{k=1}^{n-1}\frac{\la_k}{\la_k - \la_n}. \label{eqn:linear_system_of_equations_s_n}
\end{equation}%
We can write \eqref{eqn:linear_system_of_equations_s_n} in the matrix-vector notation $A \vc{s} = \vc{b}$, where we introduced the column vectors 
\begin{align}
\vc{s} &\defi ( s_i )_{i = 1,2,\ldots,N}, \text{ and } \vc{b} \defi \left( \prod_{k=1}^{n-1}\frac{\la_k}{\la_k - \la_n} \right)_{n = 1,2,\ldots,N}
\end{align}
and the lower-triangular matrix $A$ with element $(n,i)$ given by
\begin{equation}\label{eq:A_definition}%
(A)_{n,i} \defi \frac1{\rho_n^i} \prod_{k=i}^{n-1}\frac{\la_k}{\la_k - \la_n}, \quad 1 \le i \le n \le N.
\end{equation}%
We can calculate $\vc{s}$ as $\vc{s} = A^{-1}\vc{b}$. In particular, since $A$ is a lower-triangular matrix, so is its inverse $A^{-1}$. Hence, we can determine the inverse using the well-known recursive formulas
\begin{align}%
(A^{-1})_{n,n} &= \frac{1}{(A)_{n,n}} = \rho_n^n, \quad n = 1,2,\ldots,N, \\
(A^{-1})_{n,i} &= - (A^{-1})_{i,i} \sum_{k = i + 1}^n (A^{-1})_{n,k} (A)_{k,i}, \quad 1 \le i < n \le N.\label{eq:lower_triangular_inverse}
\end{align}%
This recursion is solved in a specific order. One first determines $(A^{-1})_{n,n}$, for $n = 1,2,\ldots,N$, then all $(A^{-1})_{n,n - 1}$, for $n = 2,3,\ldots,N$, followed by $(A^{-1})_{n,n - 2}$, for $n = 3,4,\ldots,N$, and so on until finally $(A^{-1})_{N,1}$ is reached. We exploit this recursion in order to derive an explicit expression for the elements of the inverse. To that end, we require some additional definitions. For any $n\in\mathbb N$ and any vector $\bm{a}=(a_{k_1},\ldots, a_{k_n}) \in (\mathbb R^+)^n$ indexed by $k_1<k_2<\ldots<k_n$ we define $M=M(\bm{a})$ to be the number of entries of the vector $\bm{a}$ that are not equal to one, i.e.,
\begin{equation}
M=M(\bm{a}) := \sum_{i=1}^n \ind{a_{k_i}\neq 1},
\end{equation}
and by $k_{(1)} < k_{(2)}<\ldots< k_{(M)}$ the ordered indices corresponding to those entries. For notational convenience, we also define $k_{(0)}\defi k_1\leq k_{(1)}$ and $k_{(M+1)} \defi k_n\geq k_{(M)}$, so that
\begin{equation*}
\bm{a} = (a_{k_{(0)}},1,\dots,1,a_{k_{(1)}},1,\dots,1,a_{k_{(2)}},1,\dots,1,a_{k_{(M-1)}},1,\dots,1,a_{k_{(M)}},1,\dots,1,a_{k_{(M+1)}}).
\end{equation*}
We then introduce the function $b: (\mathbb R^+)^n \mapsto \mathbb R$ that associates to the vector $(a_{k_1},\ldots, a_{k_n})$ the scalar $b(a_{k_1},\ldots, a_{k_n})$ defined as
\begin{align}\label{eq:b_definition_general}%
b(a_{k_1},\ldots, a_{k_n}) \defi (-1)^{M-1}\prod_{m=0}^{M} \prod_{k=k_{(m)}}^{k_{(m+1)}-1}\frac{\la_k}{\la_k - \la_{k_{(m+1)}}}.
\end{align}%
Note that, when $\la_n = \la(N - n)$,
\begin{align}\label{eq:b_definition}%
b(a_{k_1},\ldots, a_{k_n}) &= (-1)^{M-1} \prod_{m=0}^{M} \binom{N-k_{(m)}}{k_{(m+1)} - k_{(m)}}.
\end{align}%
Before proceeding, let us motivate definition \eqref{eq:b_definition_general}. We may interpret the right side of \eqref{eq:s_i_explicit_expression_main_thm} as a sum of the weights associated to Dyck paths in $\set U_i$. Then, $b(\cdot)$ represents the contribution of the up jumps to the total weight of the path $u$. The weight of each edge of the path depends on the phase where it is located, hence to compute the total weight of the path it is crucial to keep track of the location of the jumps to the right. This is accomplished by the indices $k_{(1)},\ldots, k_{(M)}$ associated to the Dyck path $u=(u_1,\ldots, u_i)$. In particular, between the $k_{(m)}$-th phase and the $k_{(m+1)}$-th phase, $u$ only makes up jumps, and then $M$ represents the total number of excursions above the diagonal of $u$. In order to prove Theorem \ref{th:main_theorem}, we first obtain an explicit expression for the inverse coefficient matrix $A^{-1}$.
\begin{proposition}\label{prop:A_inverse_explicit_expression}
Assume that $(\la_n)_{n=1}^N$ is a sequence such that ${\la_1>\ldots>\la_{N-1}>\la_N = 0}$. Then, for any $i=1,\ldots, N$ and $n = 1,2,\ldots,i - 1$ we have
\begin{equation}\label{eq:magic_formula}%
(A^{-1})_{i,i - n} = \sum_{(u_1,u_2,\ldots,u_n) \in \set{U}_n} b(\rho_{i - n}^{i - n},\rho_{i - n + 1}^{u_1}\ldots,\rho_i^{u_n}) \rho_{i - n}^{i - n} \rho_{i - n + 1}^{u_1} \rho_{i - n + 2}^{u_2} \cdots \rho_i^{u_n} ,
\end{equation}%
where $b$ was defined in \eqref{eq:b_definition_general}.
\end{proposition}
\begin{proof}
We proceed by induction, by assuming that \eqref{eq:magic_formula} holds for all $m\leq n$ for some $n\in\{1,\ldots, i-1\}$ and then proving it for $n+1$. We use \eqref{eq:lower_triangular_inverse} together with \eqref{eq:A_definition} to obtain
\begin{align}%
(A^{-1})_{i,i-(n+1)} &= -\rho_{i-(n+1)}^{i-(n+1)} \sum_{k=i-n}^{i}(A^{-1})_{i,k}(A)_{k,i-n-1} \notag\\
&= -\rho_{i-(n+1)}^{i-(n+1)} \sum_{j=0}^{n}(A^{-1})_{i,i-j}(A)_{i-j,i-n-1} \notag\\
&=-\rho_{i-(n+1)}^{i-(n+1)} \sum_{j=0}^{n}(A^{-1})_{i,i-j} \frac1{\rho_{i-j}^{i-(n+1)}}\prod_{k=i-(n+1)}^{i-j-1}\frac{\la_k}{\la_k - \la_{i-j}}.\label{eq:lower_triangular_recursion}
\end{align}%
In the last equality we highlight the inductive structure in the product term. To avoid encumbering the computations, let us denote the product in \eqref{eq:lower_triangular_recursion} as
\begin{align}%
\mathcal B_{i,j,n} \defi -\hspace{-0.3cm}\prod_{k=i-(n+1)}^{i-j-1}\frac{\la_k}{\la_k - \la_{i-j}}.
\end{align}%
Inserting the expression for $(A^{-1})_{i,i-j}$ into \eqref{eq:lower_triangular_recursion} gives
\begin{align}%
&\rho_{i-(n+1)}^{i-(n+1)}\sum_{j=0}^{n}(A^{-1})_{i,i-j} \frac1{\rho_{i-j}^{i-n-1}} \mathcal B_{i,j,n}\notag\\
&\quad= \rho_{i-(n+1)}^{i-(n+1)}\sum_{j=0}^n\sum_{(u_1,\ldots,u_j)\in\mathcal U_j}\rho_{i-j}^{i-j}\rho_{i-j+1}^{u_1}\ldots \rho_i^{u_j}b(\rho_{i - j}^{i - j},\rho_{i - j + 1}^{u_1}\ldots,\rho_i^{u_j}) \frac1{\rho_{i-j}^{i-n-1}}\mathcal B_{i,j,n}\notag\\
&\quad= \sum_{j=0}^n\sum_{(u_1,\ldots,u_j)\in\mathcal U_j}\rho_{i-(n+1)}^{i-(n+1)}\rho_{i-j}^{n+1-j}\rho_{i-j+1}^{u_1}\ldots \rho_i^{u_j}b(\rho_{i - j}^{i - j},\rho_{i - j + 1}^{u_1}\ldots,\rho_i^{u_j})\mathcal B_{i,j,n}.
\end{align}%
Now, observe that $(n+1-j) + u_1 + \ldots + u_j = n+1$. Crucially, we also have that
\begin{align}\label{eq:magic_formula_proof_inductive_step}%
&\sum_{j=0}^n\sum_{(u_1,\ldots,u_j)\in\mathcal U_j}\rho_{i-(n+1)}^{i-(n+1)}\rho_{i-j}^{n+1-j}\rho_{i-j+1}^{u_1}\ldots \rho_i^{u_j}b(\rho_{i - j}^{i - j},\rho_{i - j + 1}^{u_1}\ldots,\rho_i^{u_j}) \mathcal B_{i,j,n} \notag\\
&\quad= \sum_{(v_1, \ldots, v_{n+1})\in\mathcal U_{n+1}}\rho_{i-(n+1)}^{i-(n+1)}\rho_{i-n}^{v_1}\ldots \rho_i^{v_{n+1}}b(\rho_{i-(n+1)}^{i-(n+1)}, \rho_{i-n}^{v_1},\ldots, \rho_i^{v_{n+1}}).
\end{align}%
\begin{figure}[!bt]%
\centering%
\includegraphics[scale=0.75]{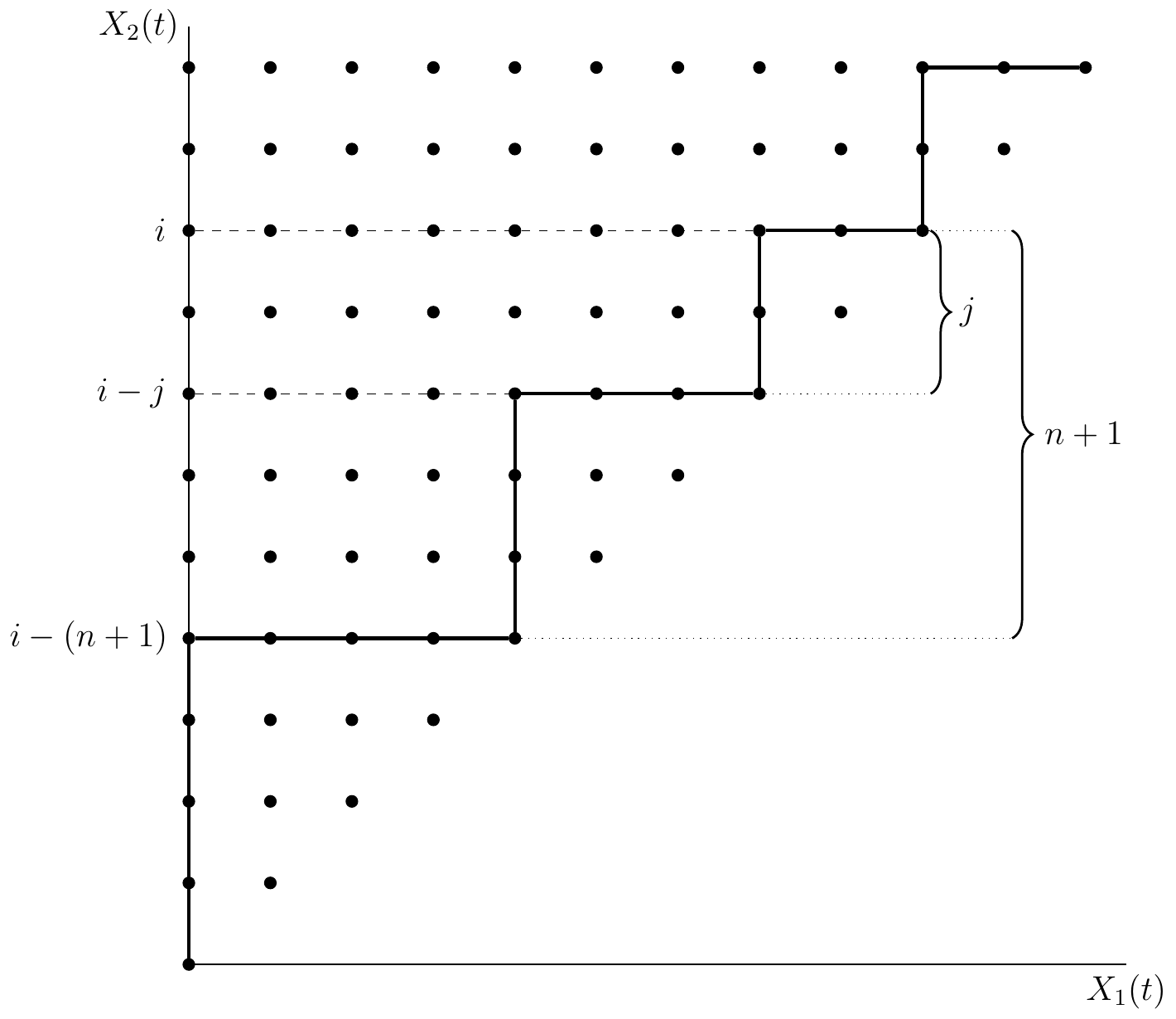}%
\caption{On the left-hand side of the inductive step \eqref{eq:magic_formula_proof_inductive_step}, the first right jump of the Dyck path occurs at phase $i-(n+1)$, and the first jump after that occurs at phase $i-j$. Summing over $j=0,\ldots, n$, one obtains all paths that jump to the right for the first time at phase $i-(n+1)$, which is the right-hand side of \eqref{eq:magic_formula_proof_inductive_step}.}%
\label{fig:Dyck_path_decomposition_example}%
\end{figure}%
Indeed, the left-hand side corresponds to the feasible assignment in which the first jump to the right occurs at phase $i-(n+1)$, which is necessarily of length $i-(n+1)$. Then, for any fixed $j=0,\ldots, n$, the next jump to the right occurs at phase $i-j$, which is necessarily of length $n+1-j$. A sum is then performed over the remaining feasible assignments. Summing over all possible $j=0,\ldots, n$ on the left-hand side of \eqref{eq:magic_formula_proof_inductive_step}, one obtains a sum over all feasible assignments such that the first jump to the right occurs at phase $i-(n+1)$, which is the sum on the right-hand side of \eqref{eq:magic_formula_proof_inductive_step}.  Furthermore, for the vector $\vc{a} = (\rho_{i-(n+1)}^{i-(n+1)}, \rho_{i-j}^{n+1-j}, \rho_{i-j+1}^{u_1},\ldots, \rho_i^{u_j})$, we have that $k_{(0)} = k_{(1)} = i-(n+1)$ and $k_{(2)}= i-j$, so that
\begin{align}%
\mathcal B_{i,j,n} = -\hspace{-0.3cm}\prod_{k=i-(n+1)}^{(i-j)-1}\frac{\la_k}{\la_k - \la_{i-j}} = -\hspace{-0.1cm}\prod_{k=k_{(1)}}^{k_{(2)}-1}\frac{\la_k}{\la_k - \la_{k_{(2)}}}.
\end{align}%
It follows that
\begin{align}%
b(\vc{a}) = b(\rho_{i-(n+1)}^{i-(n+1)}, \rho_{i-j}^{n+1-j}, \rho_{i-j+1}^{u_1},\ldots, \rho_i^{u_j}) = \mathcal B_{i,j,n}b(\rho_{i - j}^{i - j},\rho_{i - j + 1}^{u_1}\ldots,\rho_i^{u_j}) 
\end{align}%
Figure \ref{fig:Dyck_path_decomposition_example} illustrates this decomposition in terms of Dyck paths.
\end{proof}

We can finally prove Theorem \ref{th:main_theorem} by applying Proposition \ref{prop:A_inverse_explicit_expression} to invert the matrix $A$.
\begin{proof}[Proof of Theorem \ref{th:main_theorem}.]  
Writing $\vc{s} = A^{-1}\vc{b}$ explicitely yields
\begin{align}\label{eq:explicit_expression_first_hitting_time_probability}%
s_i = \sum_{n=0}^{i-1}(A^{-1})_{i,i-n}\prod_{k=1}^{i-n-1}\frac{\la_k}{\la_k - \la_{i-n}}.
\end{align}%
Plugging \eqref{eq:magic_formula} into \eqref{eq:explicit_expression_first_hitting_time_probability}, using the same inductive argument as in \eqref{eq:magic_formula_proof_inductive_step} and noting that
\begin{align}%
\prod_{k=1}^{i-n-1}\frac{\la_k}{\la_k - \la_{i-n}} = \prod_{k=k_{(0)}}^{k_{(1)}-1}\frac{\la_k}{\la_k - \la_{i-n}},
\end{align}%
gives
\begin{align}%
s_i = \sum_{(u_1,u_2,\ldots,u_i) \in \set{U}_i} \rho_{1}^{u_1} \rho_{2}^{u_2} \cdots \rho_i^{u_i} b(\rho_{1}^{u_1},\rho_{2}^{u_2},\ldots,\rho_i^{u_i}),
\end{align}%
concluding the proof.
\end{proof}

\bibliographystyle{abbrv}
\bibliography{library}

\end{document}

%% file: TikZSettingsJori.tex
\usepackage{pgfplots}
\pgfplotsset{compat = newest}
\usetikzlibrary{arrows.meta}
\usetikzlibrary{decorations.pathreplacing}
\usetikzlibrary{decorations.markings}
\usetikzlibrary{calc}
\usetikzlibrary{positioning}

\pgfdeclarelayer{arrowlayer}
\pgfsetlayers{main,arrowlayer}

\tikzset{%
    > = {Latex[width=1.7mm,length=2.2mm]}%
    }%

\tikzset{%
    dashedarrow/.style = {%
        draw, densely dashed, > = {Latex[width = 1.7mm, length = 2.2mm, open, fill = white]}, decoration = {markings, mark = at position 0.999 with {\begin{pgfonlayer}{arrowlayer} \arrow{>} \end{pgfonlayer}}},
        postaction = {decorate}
    }%
}%

\tikzset{%
    state/.style = {%
        draw, circle, minimum size = 4, inner sep = 0, fill = black
    }%
}%

\usepackage{ifthen}

%% file: MacrosJori.tex
\usepackage{amsthm}

\newtheoremstyle{italic}
{5pt}
{5pt}
{\itshape}
{}
{}
{}
{.5em}
{\bfseries{\thmname{#1}~\thmnumber{#2}.}\thmnote{~\textnormal{(#3)}}}

\newtheoremstyle{italic_borrowed}
{5pt}
{5pt}
{\itshape}
{}
{}
{}
{.5em}
{\bfseries{\thmname{#1}~\thmnumber{#2}.}\thmnote{~\textnormal{[#3]}}}

\newtheoremstyle{upright}
{5pt}
{5pt}
{\upshape}
{}
{\bfseries}
{}
{.5em}
{\bfseries{\thmname{#1}~\thmnumber{#2}.}\thmnote{~\textnormal{(\textit{#3}\textrm{)}}}}

\theoremstyle{italic}
\newtheorem{theorem}{Theorem}[section]
\newtheorem{lemma}[theorem]{Lemma}
\newtheorem{proposition}[theorem]{Proposition}
\newtheorem{corollary}[theorem]{Corollary}

\theoremstyle{italic_borrowed}

\theoremstyle{upright}

\usepackage{dsfont}
\usepackage{amsfonts}
\usepackage{bm}

\usepackage{xparse}

\usepackage{mathtools}

\usepackage{amssymb}

\newcommand{\defi}{\coloneqq}

\newcommand{\ind}[1]{\mathds{1}\{#1\}} 

\newcommand{\Nat}{\mathbb{N}} 
\newcommand{\Real}{\mathbb{R}} 
\newcommand{\Complex}{\mathbb{C}} 

\NewDocumentCommand \E { m g }{%
    \IfNoValueTF{#2}
        {\mathbb{E}[#1]}
        {\mathbb{E}_{#1}[#2]} 
    }%
\NewDocumentCommand \Esc { m g }{%
    \IfNoValueTF{#2}
        {\mathbb{E}\left[#1\right]}
        {\mathbb{E}_{#1}\!\!\left[#2\right]} 
    }%
\NewDocumentCommand \Efxd { m g }{%
    \IfNoValueTF{#2}
        {\mathbb{E}\Bigl[#1\Bigr]}
        {\mathbb{E}_{#1}\!\Bigl[#2\Bigr]} 
    }%
\NewDocumentCommand \Prob { m g }{%
    \IfNoValueTF{#2}
        {\mathbb{P}(#1)}
        {\mathbb{P}_{#1}(#2)} 
    }%
\NewDocumentCommand \Probfxd { m g }{%
    \IfNoValueTF{#2}
        {\mathbb{P}\Bigl(#1\Bigr)}
        {\mathbb{P}_{#1}\!\Bigl(#2\Bigr)} 
    }%

\newcommand{\bld}[1]{\mathbf{#1}} 
\NewDocumentCommand \eb { m g }{%
    \IfNoValueTF{#2}
        {\bld{e}_{#1}}
        {\bld{e}^{(#1)}_{#2}} 
    }%
\NewDocumentCommand \zerob { g }{%
    \IfNoValueTF{#1}
        {\bld{0}}
        {\bld{0}^{(#1)}} 
    }%

\newcommand{\la}{\lambda} 

\newcommand{\vc}[1]{\mathbf{#1}} 

\NewDocumentCommand \rtwo { m g g }{
    \IfNoValueTF{#2}
        {w_{#1}}
        {w^{(#1)}_{#2}(#3)} 
    }%

\newcommand{\stateSpace}{\mathfrak{S}} 
\newcommand{\set}[1]{\mathcal{#1}} 